\documentclass[reqno]{amsart}

\usepackage{amssymb,amsmath,amscd,amsthm}
\usepackage{cite}
\usepackage{graphicx}
\graphicspath{ {./images/} }
\usepackage{float} 
\usepackage{graphicx,psfrag}

\newtheorem{theorem}{Theorem}[section]

\newtheorem{lemma}[theorem]{Lemma}
\newtheorem{proposition}[theorem]{Proposition}

\newcommand{\Lim}[1]{\raisebox{0.5ex}{\scalebox{0.8}{$\displaystyle \lim_{#1}\;$}}}

\date{}
\begin{document}
	\title{Non-uniqueness for the ab-family of equations in Periodic case}
	\author{
		Rajan Puri }\thanks{Department
		of Mathematics and Statistics, Wake Forest University,
		Winston Salem, NC 27109, USA, (purir@wfu.edu). 
	}
	
	\date{}
	

	\begin{abstract} 
		For the cubic ab-family of equations with $a\neq0$, it is proved that there exist an initial data in the Sobolev space $H^s$, $s<3/2$, with non-unique solutions on circle. 
	\end{abstract}

	\keywords{
		Well-posedness,  
		initial value problem,  
		Cauchy problem, 
		Sobolev spaces, 
		Camassa-Holm equation,
		solitons, peakons.}
	
	\subjclass[2010]{Primary: 35Q35}
	
	\maketitle

	\section{Introduction} This paper is a continuation of our work \cite{HP} where we studied the non uniqueness for the ab-family of equations in non periodic case. The goal of this paper is to prove the same result in periodic case. We consider the Cauchy problem for the  ab- family of equations 
	\begin{align}\label{eq1}
	&u_t+u^2u_x-au_x^3+D^{-2}\partial_x\big[\frac{b}{3} u^3+\frac{6-6a-b}{2}uu_x^2\big]+D^{-2}\big[\frac{2a+b-2}{2}u_x^3\big]=0
	\\
	&u(x,0) = u_0(x), \label{eq1-data}
	\end{align}
	where $u_0(x) \in H^s (\mathbb R)$, for $s<3/2$ and $t\in \mathbb R, \ x\in \frac{\mathbb R}{2\pi \mathbb Z}$.
	The two parameters $a$, $b\in \mathbb R$ and we will assume $a\neq 0$.  $u_t$ 
	and $u_x$ denote the derivatives of $u$ with respect to $t$ and $x$, $\partial_x$ 
	denotes differentiation with respect to $x$, and the non-local operator 
	$D^{-2} = (1-\partial_x^2)^{-1}$ 
	is the inverse Fourier transform of $(1+\xi^2)^{-1} $. The local form of ab-family is given by
	\begin{equation}\label{eq111}
	\begin{split}
	(1-\partial_x^2)u_t+(b+1)u^2u_x-3au_x^3-(6a+b)uu_xu_{xx}  
	+ 6au_xu_{xx}^2-u^2u_{xxx}+3au_x^2u_{xxx}=0.  
	\end{split}  
	\end{equation}
	
	Peakon traveling wave solutions were discovered in 1978 by Fornberg and Whitham \cite{FW} and then by Camassa and Holm [CH] in their quest for a water wave model that could capture wave breaking. Peakons have discontinous spatial derivatives at their peaks so they make sense only as distributional solutions, see \cite{Ama} for details. The periodic one peakon of the ab-family (\ref{eq1}) is given by 
	$$
	u(x,t) = \pm \sqrt{c}\big(1+(1-a)\sinh^2\pi\big)^{-1/2}\cosh({[x-ct]_p-\pi}). 
	$$ 
	where $c>0$ and   where $[.]_p$ makes a quantity $2\pi$ periodic via $[x]_p=x-2\pi\lfloor{\frac{x}{2\pi}\rfloor}.$ 
	
	The well-posedness theory for the ab-family is not completely understood. Two members of ab-family (\ref{eq1}) have been studied extensively in different contexts by several researchers. In particular, the choice of parameters $a = 1/3$, $b = 2$ corresponds to the Fokas-Olver-Rosenau-Qiao (FORQ) equation derived in Fokas \cite{F}, Fuchssteiner \cite{Fu}, Olver and Rosenau \cite{OR}, and Qiao \cite{Q};
	while the choice a = 0, b = 3 gives the Novikov equation (NE) derived by Novikov \cite{N}. Himonas and Mantzavinos \cite{HMA} showed that FORQ is well-posed in $H^s$, with $s>5/2$ and this was extended to a four-parameter family in \cite{HMA2} which includes the $ab$-family. A nonuniqueness result by Himonas and Holliman \cite{AH} showed that the FORQ equation is ill-posed in $H^s$ for any $s<3/2$ for both periodic and non-periodic case. There is no theory concerning well-posedness in the gap $3/2\le s\le 5/2$. In contrast, the NE is well-posed in $H^s$ for all $s>3/2$, see \cite{BO} for details, and Himonas, Kenig and Holliman \cite{HKM} showed ill-posedness in $H^s$ for $s<3/2$. Both the ill-posedness results for the NE equation and the FORQ equation study the behavior of the solution near the time of collision of a 2-peakon solution. We use similar idea of constructing specific 2-peakon solutions and the existence of collision time to prove  non-uniqueness for the ab-family of equations in $H^s$ for all $s<3/2$ in our earlier paper \cite{HP}. Perhaps the most interesting phenomenon discovered in our paper \cite{HP}, is that multipeakon solutions to the ab-equation interact unlike classical solitons. For some values of $a$ and $b$,  we proved that our 2-peakon solutions are entangled and do not seem to separate and  smaller solitons can begin behind larger solitons, and yet, they collide. The main theorem of this paper is stated  as below.  
	\begin{theorem}\label{main-thm} For all $b \in \mathbb{R}$ and $a \neq 0$, solutions to the Cauchy problem for the ab- family of equations are not unique in $H^s$ on circle when $ s<3/2$.
	\end{theorem}
	Our proof of nonuniqueness for the periodic case follows the same strategy as used in  the real line case \cite{HP}. The paper is organized as follows. In the next section, we study the relationship between the system of ODEs with multipeakon solutions of the ab-family of equations on circle.  Later we will prove that the collision profile is indeed a single peakon solution and use this to establish the nonuniqueness. 

	\section{The ODE System}
	
	\begin{lemma}{ \cite{Ama} } The   periodic 2-peakon  function
		\begin{equation}\label{eq4}
		u(x,t)=p_1(t)\cosh([x-q_1(t)]_p-\pi)+ p_2(t)\cosh([x-q_2(t)]_p-\pi),
		\end{equation}
		solves the ab-equation on circle if and only if the positions $q_1, q_2$ and the momenta $p_1, p_2$ satisfy the $4\times 4$ system of ordinary differential equations:
		\begin{equation} \label{eq5}
		\begin{split}
		q_1^{'} & = p_1^2[1+(1-a)\sinh ^2\pi]+2p_1p_2\cosh \pi \cosh ([q_1-q_2]_p-\pi)+p_2^2[1+(1-3a)\sinh^2([q_1-q_2]_p-\pi)],\\
		q_2^{'}& = p_2^2[1+(1-a)\sinh ^2\pi]+2p_1p_2\cosh \pi \cosh ([q_2-q_1]_p-\pi)+p_1^2[1+(1-3a)\sinh^2([q_2-q_1]_p-\pi)],\\
		p_1^{'}& = (2-b)p_1p_2 \sinh([q_1-q_2]_p-\pi)\cdot [p_1\cosh \pi +p_2 \cosh([q_1-q_2]_p-\pi)],\\
		p_2^{'}& = (2-b)p_1p_2 \sinh([q_2-q_1]_p-\pi)\cdot [p_2\cosh \pi +p_1 \cosh([q_2-q_1]_p-\pi)],
		\end{split}
		\end{equation}
		where $[.]_p$ makes a quantity $2\pi$ periodic via $[x]_p=x-2\pi\lfloor{\frac{x}{2\pi}\rfloor}.$
	\end{lemma} 
	
	The solution $u$ is a $2\pi$ periodic function and for simplicity, we will restrict our attention to the interval $[0, 2 \pi].$ We define 
	$
	q(t) \  \dot= \ q_2 - q_1,
	$
	$h \  \dot= \ p_2 - p_1$, $w\dot = p_1 +p_2$ 
	and $z \  \dot= \ p_1 p_2$.
	\begin{align}\label{qhwz-system}
	\begin{aligned}[c]
	q' &=  hw\big[(1-a) \sinh^2\pi +(3a-1)\sinh^2([q]_p-\pi)\big].,\\
	h' &= (2-b) wz \sinh([q]_p-\pi)\big[\cosh(\pi)+\cosh([q]_p-\pi) \big],\\
	w' &=(2-b) hz\sinh([q]_p-\pi)\big[\cosh(\pi)-\cosh([q]_p-\pi) \big] ,\\
	z' &= -(2-b) hwz \sinh([q]_p-\pi)\cosh([q]_p-\pi) ,
	\end{aligned}
	\qquad
	\begin{aligned}[c]
	q(0)&= q_2(0) - q_1(0) =  \mu   ,\\
	h(0)&= p_2(0) - p_1(0) = h_0,\\
	w(0) &= p_1(0) + p_2(0) = w_0, \\
	z(0)&= p_1(0)  p_2(0) = z_0,
	\end{aligned}
	\end{align}
	where  $\mu, h_0, w_0, z_0$ are any real numbers. We will assume that, $q_1 (0)= 0$ 
	and $q_2(0) = \mu >0$, and therefore, by continuity,  $q(t) \ge 0$ for some 
	$0 \le t \le T^c$, where $T^c $ is the collision time. From the system of equations (\ref{eq5}), we have 
	\begin{equation}\label{eq6}
	\begin{split}
	q^{'}=q_2^{'}- q_1^{'} & =(p_2^2-p_1^2)\big[(1-a) \sinh^2\pi +(3a-1)\sinh^2([q]_p-\pi)\big].
	\end{split}
	\end{equation}
	Additionally, we define  $p(t)=p_2^2(t)-p_1^2(t)$. Now, 
	\begin{equation}\label{eq7}
	\begin{split}
	p^{'}=2p_2p_2^{'}- 2p_1p_1^{'} & =2(2-b)p_1p_2 \sinh([q_2-q_1]_p-\pi)\big[(p_1^2+p_2^2)\cosh\pi+2p_1p_2\cosh([q_2-q_1]_p-\pi)\big].
	\end{split}
	\end{equation}
	
	If  $a=1/3$, then $\big[(1-a) \sinh^2\pi +(3a-1)\sinh^2([q]_p-\pi)\big] =2/3 \sinh^2\pi$. We set $q(0)=\mu>0$ to be a small positive number. Else, we can always find a number $c, \ 1<c<2$  such that 
	$$
	0<  \mu = \mu(a) =[\mu]_p= \pi +\sinh^{-1}\bigg(\pm \sqrt{\frac{a(c+\sinh^2\pi)-1}{3a-1}}\bigg) ,
	$$
	which implies 
	$$
	L_a(\mu) \ \dot = \ \big[(1-a) \sinh^2\pi +(3a-1)\sinh^2([\mu]_p-\pi)\big] = c a .
	$$
	We note that as $ q$ tends to $0$, $\big[(1-a) \sinh^2\pi +(3a-1)\sinh^2([q]_p-\pi)\big]$ tends to $2a\sinh^2\pi$, and therefore, the sign of this term remains constant. Also, $L_a(\mu) $ can never take the value $0$ when $a\neq 0$ on the domain $ \mu \ge 0$. 
	\begin{proposition}
		If  $a \neq 0 $, the initial value problem \eqref{qhwz-system} has a unique smooth solution on some positive time interval. Furthermore, the functions $h(t )$ and $w(t)$ and $z(t)$ remain  bounded for all $ q(t)\geq 0 $.

	\end{proposition}
	\begin{proof}
		The right hand side of the  system \eqref{qhwz-system} is smooth in the arguments $q,h,w,z$ and therefore, by the ODE theorem, has a solution on some time interval $[0, T)$, $T>0$. We define $E(x)=\cosh([x]_p-\pi)$ and $E^{'}(x)=\sinh([x]_p-\pi).$ We now derive the relations between $h, w, z$ and $q$ from (\ref{qhwz-system}). 
		\begin{align}\label{qhwz1-system}
		\begin{aligned}[c]
		q' &=  hw\big[(1-a) \sinh^2\pi +(3a-1)(E^{'}(q))^2\big],\\
		h' &= (2-b) wz E^{'}(q)\big[\cosh(\pi)+E(q) \big],\\
		w' &=(2-b) hzE^{'}(q)\big[\cosh(\pi)-E(q) \big] ,\\
		z' &= -(2-b) hwz E(q)E^{'}(q).
		\end{aligned}
		\end{align}
		We follow the same strategy as in the non-periodic case.\\
		\textbf{Expressing $z$ interms of $q$:} Beginning with $z$, from the equations (\ref{qhwz1-system}) with $z'$ and $q'$ we find 
		\begin{align*}
		\frac{z'}{z} = \frac{ -(2-b)E(q)E^{'}(q)q'}
		{ \big[(1-a) \sinh^2\pi +(3a-1)(E^{'}(q))^2\big]}.
		\end{align*}
		Therefore, we have 
		\begin{align*}
		\frac{d}{dt} \ln|z|  =  \frac{(2-b)}{2(1-3a)}  \frac{d}{dt } \ln |(1-a) \sinh^2\pi +(3a-1)(E^{'}(q))^2 | .
		\end{align*}
		If $z_0 >0$, then by continuity, we can assume $z>0$ for some time and $|z| = z$. Likewise, if $z_0 < 0$, we will assume $|z| = -z$. In either case, integrating from $0$ to $t$ yields,
		\begin{align}
		\ln \left( \frac{z(t)}{z_0} \right) =  \frac{(2-b)}{2(1-3a)} \ln \left(  \frac{ (1-a) \sinh^2\pi +(3a-1)(E^{'}(q))^2  }{  (1-a) \sinh^2\pi +(3a-1)(E^{'}(\mu))^2  }  \right) .
		\end{align} 
		We exponentiate and rearrange terms to find the expression $z$ interms of $q$
		\begin{align}\label {z-eq}
		z(t) =    z_0  \left(  \frac{ (1-a) \sinh^2\pi +(3a-1)(E^{'}(q))^2  }{ (1-a) \sinh^2\pi +(3a-1)(E^{'}(\mu))^2 }  \right)^{  \frac{(2-b)}{2(1-3a)}} .
		\end{align}
		From here we can conclude that $z(t) $ remains bounded for all $ t \in [0, T^c]$, since neither the numerator nor the denominator take the value $0$. \\
		
		\textbf{Expressing $h$ interms of $q$:} Now, we  will use the above formula for $z(t)$ to find $h(t)$. We have 
		\begin{align*}
		\frac{h'}{q'} =  \frac{(2-b) wz E^{'}(q)\big[\cosh(\pi)+E(q) \big] } {  hw\big[(1-a) \sinh^2\pi +(3a-1)(E^{'}(q))^2\big]}, 
		\end{align*} 
		or rearranging we find 
		\begin{align*}
		h h'  =  z\cdot   \frac{(2-b) E^{'}(q)\big[\cosh(\pi)+E(q) \big]q^{'} } {  \big[(1-a) \sinh^2\pi +(3a-1)(E^{'}(q))^2\big]}.
		\end{align*} 
		We substitute the formula found for $z$ in equation \eqref {z-eq} to get
		\begin{align*}
		h h'  = z_0  \left(  \frac{ (1-a) \sinh^2\pi +(3a-1)(E^{'}(q))^2  }{ (1-a) \sinh^2\pi +(3a-1)(E^{'}(\mu))^2 }  \right)^{  \frac{(2-b)}{2(1-3a)}}  \frac{(2-b) E^{'}(q)\big[\cosh(\pi)+E^(q) \big]q^{'} } {  \big[(1-a) \sinh^2\pi +(3a-1)(E^{'}(q))^2\big]} .
		\end{align*} 
		We define 
		\begin{align}\label{f-def} 
		g(q) \ \dot = z_0  \left(  \frac{ (1-a) \sinh^2\pi +(3a-1)(E^{'}(q))^2  }{ (1-a) \sinh^2\pi +(3a-1)(E^{'}(\mu))^2 }  \right)^{  \frac{(2-b)}{2(1-3a)}}  \frac{(2-b) E^{'}(q) } {  \big[(1-a) \sinh^2\pi +(3a-1)(E^{'}(q))^2\big]}, 
		\end{align} 
		and then define 
		$$
		G_1 (q) \ \dot = \ \int_0^t   (\cosh(\pi)+E(q)) g(q)  dq  . 
		$$
		Since $g(q)$ is smooth and bounded for all $q \ge 0$ (since the denominator is singular only when $a = 0$), $G(q) $ remains smooth, bounded  and differentiable for all $0\le t\le T^c$. 
		Therefore, 
		$$
		h^2   = h_0^2 + 2 G_1(q) ,
		$$
		remains bounded for all $0\le t \le T^c$.\\
		\textbf{Expressing $w$ interms of $q$:} Next we solve for $w(t)$. We rearrange the equations for $w'$ and $q'$ to find 
		\begin{align*}
		ww' =  z\cdot   \frac{(2-b) E^{'}(q)\big[\cosh(\pi)-E(q) \big]q^{'} } {  \big[(1-a) \sinh^2\pi +(3a-1)(E^{'}(q))^2\big]}  = (\cosh(\pi)-E(q) ) g(q) q',
		\end{align*}
		where we used the definition of $g(q)$ found in equation \eqref{f-def} and the formula for $z$ found in equation \eqref {z-eq}. 
		Define 
		$$
		G_2 (q) \ \dot = \ \int _0^t (\cosh(\pi)-E(q) ) g(q) dq ,
		$$
		we find 
		$$
		w^2 = w_0^2 + 2 G_2 (q),
		$$
		and similarly to $h$, $w$ remains smooth and bounded for all $0\le t \le T^c$. 
	\end{proof} 
	Since $p_1 = \frac12(w-h)$ and $p_2 = \frac12(h+w)$, the above proposition shows that as long as $ q \ge 0$, $p_1, p_2 <\infty$. We will next  choose the initial values for $p_1(0)$ and $p_2(0)$ to be consistent with the above proposition and which will necessarily lead to a collision time $T^c<\infty.$ We will show that $T^c$ exists  and find an upper bound by showing $q'(t)$ is bounded by a negative number so long as $q(t)$ remains non-negative. More precisely, we will prove the following Theorem.
	\begin{theorem} {\label{thm2}}For all $b \in \mathbb {R} $ and $a\neq 0$, there exists an initial multipeakon profile on circle such that for some $\epsilon>0$,  $\frac{dq}{dt} < -\epsilon  <0$ for $t\in [0, T^c)$ (and hence $T^c<\infty$) or there exists a time $T^p$ such that at least one of $p_1(T^p)  =0$ or $p_2(T^p) = 0$. 
	\end{theorem} 
	
	\begin{proof} If $\min\{T^p, T^c\} = T^p$ and $T^p<\infty$, we are done. Therefore, we will assume $p_1(t)$ and $p_2(t)$ do not equal zero. By continuity, whatever the sign of their initial data is, we may assume the solutions take as well. 
		We will prove the theorem in four cases, based upon the values of $a$ and $b$, omitting the trivial case when $b=2$. For each case, we will consider the function $p(t)$ and $q(t)$. From (\ref{eq6}) and (\ref{eq7}), 
		\begin{align}\label{eq110} 
		q'(t) &=p(t)\big[(1-a) \sinh^2\pi +(3a-1)\sinh^2([q]_p-\pi)\big],
		\\
		\label{eq10}
		p'(t)&=2(b-2)p_1p_2 \sinh(\pi-[q]_p)\big[(p_1^2+p_2^2)\cosh\pi+2p_1p_2\cosh([q]_p-\pi)\big].
		\end{align}
		
		\noindent{\bf Case 1: $ a>0, b>2$.}
		We take the initial data 
		$$
		p_1(0)=\alpha + \delta,\ q_1(0)=0, \ p_2(0)= -\alpha,\ q_2(0)=\mu.
		$$
		
		By the choice of our initial data: 
		$$p(0)=p_2^2(0)-p_1^2(0)= -(2\alpha \delta + \delta^2)<0, \text{ and } q(0)=q_2(0)-q_1(0)= \mu >0. $$ 
		Since $p_2(0)=-\alpha <0$ and $p_1(0)= \alpha + \delta >0, \ p_1(0)p_2(0) <0$ and by continuity $p_1(t)p_2(t) <0$.

		We will now show that for all $0\le t \le T^c$,  
		$$\frac{dp}{dt}=2(b-2)p_1p_2 \sinh(\pi-[q]_p)\big[(p_1^2+p_2^2)\cosh\pi+2p_1p_2\cosh([q]_p-\pi)\big]<0.$$
		Indeed, the following calculation shows that  $\big[(p_1^2+p_2^2)\cosh\pi+2p_1p_2\cosh([q]_p-\pi)\big]>0.$
		Use the fact that  $\cosh\pi > \cosh([q]_p-\pi)$    
		to compute 
		\begin{equation}
		\begin{split}
		\big[(p_1^2+p_2^2)\cosh\pi+2p_1p_2\cosh([q]_p-\pi)\big]\geq 
		& (p_1^2+p_2^2)\cosh([q]_p-\pi)+2p_1p_2 \cosh([q]_p-\pi)\\ 
		& = (p_1+p_2)^2\cosh([q]_p-\pi)\geq 0.
		\end{split}
		\end{equation}
		We notice that for all $0\le t \le T^c$ either $ \sinh(\pi-[q]_p)\geq0$ or $\sinh(\pi-[q]_p)\leq0.$  \\
		\textbf{Remarks}: We may assume the case with $ \sinh(\pi-[q]_p)\geq0$. Otherwise we may need to choose the different initial profile such that $\frac{dp}{dt}$ becomes negative by having $p_1p_2>0.$ That means we need to treat the case differently by having a initial profile (similar to the case 2 below).$$
		p_1(0)=\alpha +\delta,\ q_1(0)=0, \ p_2(0)= \alpha ,\ q_2(0)=\mu.
		$$ Either way, our method works and gives the same result. \\[.1in]
		For simplicity, we may assume $ \sinh(\pi-[q]_p)\geq0$. Therefore, $\frac{dp}{dt} <0$. Now, we may use the fact that $p(t) < p(0)<0$.  Substituting this into  equation (\ref{eq110}), we have 
		$$\frac{dq}{dt} \leq  p(0)\big[(1-a) \sinh^2\pi +(3a-1)\sinh^2([q]_p-\pi)\big] .$$
		The right hand side is negative, since initially $0< a \le L_a(\mu)  = ca \le 2a$, and as $q  $ decreases, 
		$$
		ca \le\big[(1-a) \sinh^2\pi +(3a-1)\sinh^2([q]_p-\pi)\big]  \le 2a. 
		$$
		Therefore we compute $\epsilon$:
		$$\frac{dq}{dt} \leq a p(0) = -a(2\alpha \delta + \delta^2) =-\epsilon <0 $$  Hence, either $p_1(t) = 0$, $p_2(t) = 0$, or $q(t) = 0$ in finite time.  \\[.1in]
		\textbf{Note: } We assume that $ \sinh(\pi-[q]_p)\geq0$ for all our cases. We can proof the theorem similarly with the case $ \sinh(\pi-[q]_p)\leq0$  but only difference will be to choose different initial profile as described above in the case 1. That means there exist an initial profile which guarantees the statement of the theorem. 
		
		
		\noindent{\bf Case 2: $a>0, b<2 $.} For this case, we take the two peakon initial profile:
		$$
		p_1(0)=\alpha +\delta,\ q_1(0)=0, \ p_2(0)= \alpha ,\ q_2(0)=\mu.
		$$
		By the choice of our initial data 
		$$p(0)=p_2^2(0)-p_1^2(0)= -(2\alpha \delta + \delta^2)<0, \ q(0)=q_2(0)-q_1(0)= \mu >0. $$  
		Since $ p_1(0)p_2(0) >0$, by continuity $p_1(t)p_2(t) >0$ for $t\in [0, T^c)$. We have 
		$$\frac{dp}{dt}=2(b-2)p_1p_2 \sinh(\pi-[q]_p)\big[(p_1^2+p_2^2)\cosh\pi+2p_1p_2\cosh([q]_p-\pi)\big]<0.$$
		Thus,   $p(t)<p(0) <0 \ \forall t\in [0, T^c).$ Now from the equation (\ref{eq110}),
		$$\frac{dq(t)}{dt} \leq  p(0)\big[(1-a) \sinh^2\pi +(3a-1)\sinh^2([q]_p-\pi)\big]  \ \text{for} \ t\in [0, T^{c}],$$
		and, similarly to the first case
		$$\frac{dq(t)}{dt} \leq a p(0) = -a(2\alpha \delta + \delta^2)  =-\epsilon  <0.
		$$ 
		Thus, we have shown that there exists an initial profile such that either $p_1(t) = 0$, $p_2(t) = 0$, or $q(t) = 0$ in finite time.

		
		\noindent{\bf Case 3: $a<0, b>2$.}
		Similar to Case 2, we take the two peakon initial profile:
		$$
		p_1(0)=\alpha ,\ q_1(0)=0, \ p_2(0)= \alpha +\delta,\ q_2(0)=\mu.
		$$
		By the choice of our initial data 
		$$p(0)=p_2^2(0)-p_1^2(0)= (2\alpha \delta + \delta^2)>0, \ q(0)=q_2(0)-q_1(0)= \mu >0. $$
		Also, by continuity $p_1(t)p_2(t) >0$. We will assume this hold for $t \in [0, T^c)$, thus
		$$\frac{dp}{dt}=2(b-2)p_1p_2 \sinh(\pi-[q]_p)\big[(p_1^2+p_2^2)\cosh\pi+2p_1p_2\cosh([q]_p-\pi)\big]>0.$$
		and therefore,   $p(t)>p(0)>0 $.

		Recalling the choice of $\mu$, we have $ L_a(\mu) =ca<0$ and $\Lim{q\to \ 0}L_a(q)=2a <0.$ 
		Hence, 
		\begin{equation}\label{l-est}
		2a <L_a(q(t)) < ca <0, t\in [0, T^c) .
		\end{equation}
		Therefore,
		using equation (\ref{eq110}) we have
		$$\frac{dq(t)}{dt}=  p(t)L_a(q(t) )  \leq  p(0)\cdot ca = -\epsilon<0 .
		$$
		Again, this shows that there exists an initial profile such that $p_1(t) = 0$, $p_2(t) = 0$, or $q(t) = 0$ in finite time.  
		

		\noindent{\bf Case 4: $a<0, b<2$.}
		Similar to Case 1, we take a  peakon-antipeakon initial profile:
		$$
		p_1(0)=-\alpha ,\ q_1(0)=0, \ p_2(0)= \alpha +\delta,\ q_2(0)=\mu,
		$$
		and as before, we will assume $T^c \le T^p$. By the choice of our initial data $$p(0)=p_2^2(0)-p_1^2(0)= (2\alpha \delta + \delta^2)>0, \text{ and }q(0)=q_1(0)-q_2(0)= \mu >0. $$ 
		Since $ p_1(0)p_2(0) <0$, by continuity $p_1(t)p_2(t) <0$, and therefore an argument similar to the argument presented in Case 1 shows 
		$$\frac{dp}{dt}=2(b-2)p_1p_2 \sinh(\pi-[q]_p)\big[(p_1^2+p_2^2)\cosh\pi+2p_1p_2\cosh([q]_p-\pi)\big]>0,$$
		thus $p(t)>p(0)>0 \ \forall \ t\in [0, T^c).$ 
		Therefore, using equation \eqref{eq110} and the estimate in inequality \eqref{l-est} we have again
		$$\frac{dq(t)}{dt}= p(t)L_a(q(t)) \leq  a c \cdot p(0)   = -\epsilon<0 .$$ 
		This completes the fourth case, and we have shown that for every choice of $a$ and $b$, there is an initial profile which leads to  $\min\{T^p, T^c\} < \infty$. 
	\end{proof} 
	
	\textbf{Remarks:} We can actually estimate the time of collision $T^c$ by using above theorem (\ref{thm2}) and it is given by $$T^c \leq \frac{\mu}{\epsilon}.$$
	
	\subsection{Proof of Theorem \ref{main-thm}}
	To complete the proof of Theorem \ref{main-thm}, we will show that at time, $T = \min\{T^c, T^p\}$, the solution to the Cauchy problem of ab-family of equations \eqref{eq1}-\eqref{eq1-data} with the initial profiles depending upon $a$ and $b$, is either a single peakon or the zero solution. We define the collision function: 
	\begin{equation}\label{eq18}
	C(x)=p^*\cosh([q ^*]_p-\pi),
	\end{equation}
	where if $T = T^c$, $q ^* = \Lim{t\to \ T^-} q_1(t) $ is the location of the collision and $p^{*}= \Lim{t\to \ T^-}(p_1(t)+p_2(t))$ is the magnitude of the collision. We define $q ^* = \Lim{t\to \ T^-} q_j(t) $. If $ T = T^p$, then  $\Lim{t\to \ T^-} p_i(t) = 0$, and $ \Lim{t\to \ T^-}p_j (t) \neq0$ for $i,j \in \{1,2\}$.  The choice in $q^*$ is irrelavent if both $p_1$ and $p_2$ converge to zero at time $T$, since $C(x) = 0$. 
	Now we will show that the solution converges to the collision function as shown in the following lemma. 
	\begin{lemma}
		The $H^{s}$  limit of $u$, as $t$ approaches $T$ from below is C:
		$$\lim_{t\to \ T^{-}}||u(t)-C||_{H^{s}} =0.$$ 
	\end{lemma}
	\begin{proof}
		We take the Fourier transform of $u$, and we have
		$$ \widehat{u}(\xi,t)= \sinh \pi \bigg(\frac{2p_1e^{-i\xi q_1}}{1+\xi^2}+\frac{2p_2e^{-i\xi q_2}}{1+\xi^2}\bigg).
		$$
		Similarly, we can find the Fourier transform of $C$ as
		$$ \widehat{C}(\xi)= \sinh \pi \frac{2p^{*}e^{-i\xi q^*}}{1+\xi^2}.
		$$
		Calculating the $H^{s}$ norm of $u(t)-C$ gives us
		$$\lim_{t\to \ T^{-}}||u(t)-C||_{H^{s}} ^2= 4\sinh^2 \pi \lim_{t\to \ T^{-}} \sum_{n\in Z} ({1+\xi^2})^{s-2} \big|p_1e^{-i\xi q_1}+p_2e^{-i\xi q_2}-p^{*}e^{-i\xi q^*}\big|^2.$$
		We can bound the quantity inside the absolute value by $(|p_1|+|p_2|+|p^*|) \leq M < \infty.$ Let $v(\xi)=(1+\xi^2)^{s-2}\cdot M^2$ then $v$ dominates our original summand and $v$ is itself summable when $s<3/2.$ Therefore, we may apply the Dominated Convergence Theorem and bring the limit inside the summation. 
		\begin{equation} 
		\lim_{t\to \ T^{-}}||u(t)-C||_{H^{s}} ^2=
		4\sum_{n\in Z} ({1+\xi^2})^{s-2} |p_1(T)e^{-i\xi q_1(T)}+p_2(T)e^{-i\xi q_2(T)}-p^*e^{-i\xi q^*}|^2.
		\end{equation}   
		By definition of $p^*$ and $q^*$,  the term inside the integral is zero. 
	\end{proof} 
	\section*{Acknowledgement}  
	The author like to express his sincere appreciation to John Holmes for his valuable comments and discussions.

	\bibliographystyle{plain}
	\bibliography{periodic_case}
	
\end{document}